\documentclass[12pt,reqno]{amsart}
\usepackage[a4paper,margin=2.5cm,top=2.5cm,bottom=2.5cm,centering,headheight=3ex,headsep=4ex,vcentering]{geometry}

\usepackage{orcidlink}

\usepackage[cal=cm,scr=euler]{mathalfa}
\usepackage{microtype}
\usepackage{mlmodern}
\usepackage{amsfonts,amsmath,amssymb,amsthm,hyperref,setspace,verbatim,longtable}

\usepackage{xcolor}

\newcommand\K{\rule{0pt}{2.6ex}}
\newcommand\D{\rule[-1.2ex]{0pt}{0pt}}

\usepackage[normalem]{ulem}

\usepackage{color}
\definecolor{ao(english)}{rgb}{0.0, 0.0, 1.0}

\hypersetup{colorlinks=true, linkcolor=ao(english),citecolor=ao(english)}

\newcommand{\eo}{\mathcal{EO}}
\newcommand{\eob}{\overline{\mathcal{EO}}}
\newcommand{\eou}{\mathcal{EO}_u}

\newtheorem{theorem}{Theorem}[section]
\newtheorem{conjecture}{Conjecture}[section]
\newtheorem{corollary}{Corollary}[section]
\newtheorem{lemma}{Lemma}[section]

\numberwithin{equation}{section}

\title{Congruences for the Andrews--Uncu partition function $\mathcal{EO}_u(n)$}

\author[N. D. Baruah]{Nayandeep Deka Baruah\, \orcidlink{0000-0003-1129-2929}}
\address[N. D. Baruah]{Department of Mathematical Sciences, Tezpur University, Napaam, Assam 784028, India}
\email{nayan@tezu.ernet.in}
\author[H. Das]{Hirakjyoti Das\, \orcidlink{0000-0002-5540-5625}}
\address[H Das]{Department of Mathematics, B. Borooah College (Autonomous), Guwahati 781007, Assam, India}
\email{hirak@bborooahcollege.ac.in}

\author[M. P. Saikia]{Manjil P. Saikia\,\orcidlink{0000-0002-2997-6731}}
\address[M. P. Saikia]{Mathematical and Physical Sciences division, School of Arts \& Sciences, Ahmedabad University, Navrangpura, Ahmedabad 380009, Gujarat, India}
\email{manjil.saikia@ahduni.edu.in}

\author[A. Sarma]{Abhishek Sarma\,\orcidlink{0009-0005-0075-8000}}
\address[A. Sarma]{Department of Basic Sciences and Humanities, Assam Skill University,  Mangaldai 784125, Assam, India}
\email{abhitezu002@gmail.com}



\allowdisplaybreaks
\begin{document}
	\maketitle

	\begin{abstract}
		In 2018, Andrews defined and studied the partition function $\mathcal{EO}(n)$, which counts the number of partitions of $n$ where each even part is less than each odd part. Thereafter, Uncu considered a different subset of such partitions, namely those partitions counted by $\mathcal{EO}(n)$ where there are no repeated even parts (we denote the number of such partitions by $\mathcal{EO}_u(n)$). In this paper, we prove several congruences modulo powers of $2$ for $\mathcal{EO}_u(n)$. We also prove some infinite families of congruences as well as a recurrence for the number of this class of partitions. Our methods involve elementary, algorithmic, and modular form techniques.
	\end{abstract}

	\section{Introduction}
	A partition of a positive integer $n$ is a finite non-increasing sequence of positive integers $\lambda=(\lambda_1, \lambda_2, \ldots, \lambda_k)$ such that $\lambda_1+ \lambda_2+ \cdots+ \lambda_k=n$. For instance, $(4),(3, 1), (2, 2), (2, 1, 1)$ and $(1, 1, 1, 1)$ are the five partitions of $4$. The number of partitions of $n$ is denoted by $p(n)$, and with the convention that $p(0)=1$, its generating function is given by
	\begin{align}
		\sum_{n= 0}^\infty p(n)q^n=\frac{1}{(q;q)_\infty},\label{p(n)}
	\end{align}
	where
	\[
	(a;q)_\infty:=\prod_{i= 0}^\infty (1-aq^i), \quad |q|<1.
	\] Throughout the paper, we use the notation $f_{k} := (q^k;q^k)_{\infty}$. Over the last several decades, various mathematicians have studied many different aspects of partitions as well as several restricted classes of partitions. For a general survey of the theory of partitions, we refer the reader to the books of Andrews \cite{AndrewsBook} and Johnson \cite{Johnson}.
	
	A recent example of studies in restricted classes of partitions are the partitions whose parts are separated by parity. These partitions were first studied by Andrews \cite{Andrews, AndrewsComb}, which thenceforth have been studied by several other mathematicians. Let $\eo(n)$ denote the number of partitions of $n$ in which each even part is less than each odd
	part. Further, let $\eob(n)$ denote the number of partitions of $n$ counted by $\eo(n)$ in which only the largest part even part appears an odd number of times. Andrews \cite[Corollary 3.2]{Andrews} proved that the generating function of $\eob(n)$ is given by
	\begin{equation}\label{gf:eob}
		\sum_{n=0}^\infty \eob(n)q^n=\dfrac{f_4^3}{f_2^2}.
	\end{equation}
	By looking at a different subset of $\eo(n)$, Uncu \cite{Uncu} considered the partition numbers $\eou(n)$ whose generating function is given by
	\begin{equation}\label{gf:eou}
		\sum_{n=0}^\infty \eou(n)q^n=\dfrac{f_4^2}{f_2^2}.
	\end{equation}
	From the generating function it is clear that $\eou(n)=0$ when $n$ is odd, so we restrict to the subset of the partitions counted by $\eo(2n)$. The function $\eou(2n)$ gives the number the partitions counted by $\eo(2n)$, where there are no repeated even parts. 
	
	One of the central themes in the study of integer partitions are their congruence properties. The impetus for this comes from Ramanujan's famous congruences of $p(n)$ modulo $5, 7$ and $11$. For the partitions defined above, Andrews \cite[Eq. (1.6)]{Andrews} proved that for all $n\geq 1$, we have  \[\eob(10n+8)\equiv 0 \pmod 5,\] while Ray and Barman \cite[Theorem 1.3]{RayBarman} further proved by using an algorithm due to Radu \cite{Radu2} that for all $n\geq 1$, we have
	\[
	\eob(50n+r)\equiv  0 \pmod{20} \quad \text{for}\quad r\in\{18,28,38,48\}.
	\] They also proved other results using the theory of modular forms. This was followed by other mathematicians' work on the $\eob(n)$ partition function, such as those of Guadalupe \cite{Guadalupe}, Garvan and Morrow \cite{GarvanMorrow} and the last author \cite{Sarma}. However, very less attention has been devoted to the $\eou(n)$ partition function. The objective of this paper is to remedy this.

	Goswami and Jha \cite{GoswamiJha} proved several congruences for $\eo(n)$, $\eob(n)$, and $\eou(n)$ using the theory of modular forms and  $2$-dissection formulas from Ramanujan's notebook. They also found some generating functions for $\eou(n)$, which were also proved by Pore and Fathima \cite{PoreFathima}. As an example, Goswami and Jha proved, for all $n\geq 1$, we have
	\begin{align*}
		\eob(4n+2)&\equiv 0 \pmod 2, \\
		\eou(4n+2)&\equiv 0 \pmod 2.
	\end{align*}
	At the end of the paper, they conjectured the following congruences.
	\begin{conjecture}\label{conj}
		For all $n\geq 1$, we have
		\begin{align}\label{eob:cong-conj}
			\eob(10n+r)&\equiv 0 \pmod 2,\quad r\in\{2,4\},\\
			\label{eou:cong-conj}
			\eou(10n+r)&\equiv 0 \pmod 2,\quad r\in\{2,6\}.
		\end{align}
	\end{conjecture}
	\noindent In Section \ref{Elementary Approach} (Theorem \ref{thm-2}), we prove that this conjecture is true.
	
	In another recent work, Rahman and N. Saikia \cite{RahmanSaikia} proved two infinite families of congruences modulo $2$ and $4$ for $\eou(n)$. It should be noted that the congruences found by Goswami and Jha \cite{GoswamiJha} and Pore and Fathima \cite{PoreFathima} were all modulo $2$. This paper aims to prove several Ramanujan-type congruences for $\eou(n)$ modulo $2, 4, 16$, $32$, and $64$ as well as some other results. For example, besides confirming Conjecture \ref{conj} using $q$-series, we find the generating function of $\eou(10n+4)$ (see Theorem \ref{thm:gf-14-8}) and  prove that (see Corollary \ref{thm:inf-1}) for all $n\geq 0$ and $k\ge0$, we have
	\begin{align*}
		\eou\left(2\times 5^{2k+2}n+2\times 5^{2k+1}r+\dfrac{5^{2k+2}-1}{6}\right)&\equiv  0 \pmod{16},\quad r\in\{1,2,3,4\}.
	\end{align*}
	For another example, using modular forms, we derive (see Theorem \ref{Thm SS EOu 4n+2}) that
	\begin{align*}
		\sum_{n=0}^\infty \eou (116n+82)q^n\equiv q^{15}\sum_{n=0}^\infty \eou (4n+2)q^{29n}\pmod{4}.
	\end{align*}

	The paper is organized as follows: in the next section, we refer to some preliminaries of three different techniques, namely elementary $q$-series, modular form, and algorithmic approaches, through which, we obtain our results. Sections \ref{Elementary Approach}--\ref{Algorithmic Approach} contain our results along with their proofs according to our approaches. We end the paper with some concluding remarks in Section \ref{sec:conc}.

	\section{Preliminaries}
	\subsection{Preliminaries to the Elementary Approach}
	Ramanujan's theta function \cite[p. 34, (18.1)]{BerndtIII} is defined by
	\[
	f(a,b)=\sum_{n=-\infty}^\infty a^{n(n+1)/2}b^{n(n-1)/2}, \quad |ab|<1.
	\]
	A  special case of $f(a,b)$, which comes in our use, is
	\begin{align}
		\psi(q)&=f(q,q^3)=\sum_{n= 0}^\infty q^{n(n+1)/2}=(q^2;q^2)_\infty (-q;q)_\infty.\label{psi}
	\end{align}
	We need the following two identities from \cite[Theorem 3.1]{NDB}:
	\begin{align}
		\label{bbo1} \psi^2(q)-q\psi^2(q^5)&=  \dfrac{f_2f_5^3}{f_1f_{10}},\\
		\psi^2(q)-5q\psi^2(q^5)&= \dfrac{f_1^3f_{10}}{f_2f_5}.\notag
	\end{align}
	Multiplying these identities, we get
	\begin{equation}\label{eq:ndb-1}
		\psi^4(q)-6q\psi^2(q)\psi^2(q^5)+5q^2\psi^4(q^5)=f_1^2f_5^2.
	\end{equation}
	
	We also require the following 
	identities \cite[ (2.6) and (2.29)]{NDBBegum}:
	\begin{align}\label{eq:nilufar-2}
		\frac{f_5}{f_2^2f_{10}}&=\frac{f_5^5}{f_1^4f_{10}^3}-4q\frac{f_{10}^2}{f_1^3f_2},\\
		\label{eq:nilufar-1}
		\frac{f_2^3 f_5^2}{f_1^5 f_{10}^2}&=\frac{f_5}{f_2^2f_{10}} +5q\frac{f_{10}^2}{f_1^3f_2}.
	\end{align}
	
	In what follows next in this subsection, we have dissection identities of some $q$-products. First, we have a 2-dissection from  \cite{matching}:
	\begin{align}
		\frac{1}{f_{1}^2} &= \frac{f_{8}^5}{f_{2}^5 f_{16}^2} + 2 q \frac{f_{4}^2 f_{16}^2 }{f_{2}^5 f_{8}}.\label{disf1}
	\end{align}
	Next, the 5-dissection of $f_1$ first stated by Ramanujan \cite[p. 212]{Ramanujan} and proved by Watson \cite{Watson} and the 5-dissection of $\psi(q)$ from  \cite[p. 49]{BerndtIII}:
	\begin{align}
		\label{5-diss}
		f_1&=f_{25}\left(\frac{1}{R(q^5)}-q-q^2R(q^5)\right),\\
		\label{5-diss-psi}\psi(q)&=f\left(q^{10},q^{15}\right)+qf\left(q^5,q^{20}\right)+q^3\psi\left(q^{25}\right),
	\end{align}
	where 
	$\displaystyle{R(q)=\frac{(q;q^5)_\infty (q^4;q^5)_\infty}{(q^2;q^5)_\infty(q^3;q^5)_\infty}}$ is the $q$-product representation of  Rogers-Ramanujan continued fraction $R(q)$.
	
	We also have the $7$-dissection of $f_1$ from  \cite[p. 303, Entry 17(v)]{BerndtIII}:
	\begin{equation}\label{eq:7-d}
		f_1=f_{49} (A^\prime(q^7)-qB^\prime(q^7)-q^2+q^5C^\prime(q^7)).
	\end{equation}
	where
	\[
	A^\prime(q^7)=\frac{f(-q^{14},-q^{35})}{f(-q^7,-q^{42})}, \quad B^\prime(q^7)=\frac{f(-q^{21},-q^{28})}{f(-q^{14},-q^{35})}, \quad C^\prime(q^7)=\frac{f(-q^7,-q^{42})}{f(-q^{21},-q^{28})}.
	\]
	
	Note that \eqref{5-diss} and \eqref{eq:7-d} are special cases of a result of Ramanathan \cite[Theorem 1]{Ramanathan} (also independently proved by Evans \cite{Evans}), which is stated below.
	
	\begin{theorem}\cite[Theorem 12.1]{BerndtIII}\label{lem:1}
		Let $n$ be a natural number with $n\equiv \pm 1 \pmod 6$ and let $g\geq 1$. If $n=6g+1$, then
		\begin{equation}
			f_1= f_{n^2}\Bigg((-1)^gq^{(n^2-1)/24}+\sum_{k=1}^{(n-1)/2}(-1)^{k+g}q^{(k-g)(3k-3g-1)/2}\frac{f(-q^{2nk},-q^{n^2-2nk})}{f(-q^{nk},-q^{n^2-nk})} \Bigg).
		\end{equation}
		If $n=6g-1$, then
		\begin{equation}
			f_1= f_{n^2}\Bigg((-1)^gq^{(n^2-1)/24}+\sum_{k=1}^{(n-1)/2}(-1)^{k+g}q^{(k-g)(3k-3g+1)/2}\frac{f(-q^{2nk},-q^{n^2-2nk})}{f(-q^{nk},-q^{n^2-nk})} \Bigg). 
		\end{equation}
	\end{theorem}
	\noindent A direct application of the multinomial theorem on the above result gives us the following lemma.
	\begin{lemma}\label{lem:2}
		Let $n$ be a natural number with $n\equiv \pm 1 \pmod 6$ and let $g\geq 1$. If $n=6g+1$, then
		\[
		f_1^2 \equiv f_{n^2}^2\Bigg(q^{(n^2-1)/12}+\sum_{k=1}^{(n-1)/2}q^{(k-g)(3k-3g-1)}\frac{f(-q^{2nk},-q^{n^2-2nk})^2}{f(-q^{nk},-q^{n^2-nk})^2}\Bigg) \pmod 2,
		\]
		If $n=6g-1$, then
		\[
		f_1^2\equiv  f_{n^2}^2\Bigg(q^{(n^2-1)/12}+\sum_{k=1}^{(n-1)/2}q^{(k-g)(3k-3g+1)}\frac{f(-q^{2nk},-q^{n^2-2nk})^2}{f(-q^{nk},-q^{n^2-nk})^2}\Bigg) \pmod 2.
		\]
	\end{lemma}

	\subsection{Preliminaries to the Modular Form Approach}
	This subsection contains the basic definitions and facts concerning modular forms.
	We need the following matrix groups: 
	\begin{align*}
		\textup{GL}_2^{+}(\mathbb{R})&:=\left\{\begin{pmatrix}
			a & b\\
			c & d
		\end{pmatrix}: a,b,c,d\in \mathbb{R}, ad-bc>0\right\},\\
		\textup{SL}_2(\mathbb{Z})&:=\left\{\begin{pmatrix}
			a & b\\
			c & d
		\end{pmatrix}: a,b,c,d\in \mathbb{Z}, ad-bc=1\right\},\\
		\Gamma_0(N)&=\left\{\begin{pmatrix}
			a & b\\
			c & d
		\end{pmatrix}\in \textup{SL}_2(\mathbb{Z}): c\equiv 0\pmod{N}\right\},\\
		\Gamma_1(N)&=\left\{\begin{pmatrix}
			a & b\\
			c & d
		\end{pmatrix}\in \Gamma_0(N): a\equiv d\equiv 1\pmod{N}\right\},\\
		\Gamma(N)&:=\left\{\begin{pmatrix}
			a & b\\
			c & d
		\end{pmatrix}\in \textup{SL}_2(\mathbb{Z}): a\equiv d\equiv 1\pmod{N}, b\equiv c\equiv 0\pmod{N}\right\},
	\end{align*}
	where $N$ is a positive integer.
	
	A  subgroup $\Gamma$ of $\textup{SL}_2(\mathbb{Z})$ is called a congruence subgroup  if $\Gamma(N)\subseteq \Gamma$. The smallest such $N$ is called the level of $\Gamma$. The upper half of the complex plane is denoted by $\mathbb{H}:=\{z\in\mathbb{C}: \textup{Im}(z)>0\}$. The action of $\textup{GL}_2^{+}(\mathbb{R})$ on $\mathbb{H}$ is given by $\begin{pmatrix}
		a & b\\
		c & d
	\end{pmatrix}z=\dfrac{az+b}{cz+d}$. We denote the extended upper half plane by $\mathbb{H}^\star =\mathbb{H}\cup \mathbb{Q}\cup \{\infty\}$, where we identify $\infty:=\dfrac{1}{0}$.  The action of $\textup{GL}_2^{+}(\mathbb{R})$ on $\mathbb{H}^\star$ is given by $\begin{pmatrix}
		a & b\\
		c & d
	\end{pmatrix}\dfrac{r}{s}=\dfrac{ar+bs}{cr+ds}$, where $\dfrac{r}{s}\in\mathbb{Q}\cup \{\infty\}$. A cusp of a congruence subgroup $\Gamma$ of $\textup{SL}_2(\mathbb{Z})$ is an equivalence class in $\mathbb{Q}\cup \{\infty\}$ under the action of $\Gamma$.  The slash operator $\mid_\ell$ for integer $\ell$ and $\gamma=\begin{pmatrix}
		a & b\\
		c & d
	\end{pmatrix}\in \textup{GL}_2^{+}(\mathbb{R})$ on a meromorphic function $f(z)$ in $\mathbb{H}$  is defined by
	\begin{align*}
		\left(f\mid_\ell \gamma\right)(z):= \left(\textup{Det}\gamma\right)^{\ell/2}(cz+d)^{-\ell}f\left(\gamma z\right).
	\end{align*}
	
	A modular form with weight $\ell$ on a congruence subgroup $\Gamma$ of level $N$ is a holomorphic function $f:\mathbb{H}\to \mathbb{C}$ such that for $\begin{pmatrix}
		a & b\\
		c & d
	\end{pmatrix}\in \Gamma$ and $\gamma\in \textup{SL}_2(\mathbb{Z})$, we have
	\begin{align*}
		f\left(\dfrac{az+b}{cz+d}\right)&=(cz+d)^\ell f(z), &
		\left(f\mid_\ell \gamma\right)(z)&=\sum_{n=0}^\infty a_\gamma (n) q_N^n, 
	\end{align*} where $q_N:=e^{2\pi iz/N}$.

	Next, we recall the definition of the Nebentypus character $\chi$ from \cite{Ono04}.  If $\chi$ is a Dirichlet character modulo $N$, then a modular form $f(z)$ of weight $\ell$ with respect to $\Gamma_1(N)$ is said to have the Nebentypus character $\chi$ if for all $\begin{pmatrix}
		a & b\\
		c & d
	\end{pmatrix}\in \Gamma_0(N)$,
	\begin{align*}
		f\left(\dfrac{az+b}{cz+d}\right)&=\chi(d)(cz+d)^\ell f(z).
	\end{align*}
	The space of such modular forms with a  Nebentypus character is denoted by $M_{\ell}(\Gamma_0(N),\chi)$.
	
	Now, for determining the spaces $M_{\ell}(\Gamma_0(N),\chi)$ of certain $\eta$-quotients, we need two lemmas from \cite{Ono04} stated below. The function $\eta(z)$ is defined as $\eta(z):=q^{1/24}\left(q;q\right)_\infty,$ where $z\in \mathbb{H}$ and  $q:=e^{2\pi i z}$. Hereafter in this subsection and Section \ref{Modular Approach}, we take $\eta_n:=\eta(nz)$ for integers $n\ge1$.

	\begin{lemma}\cite[Theorem 1.64]{Ono04}\label{Modularity of Eta-Quotients} For the $\eta$-quotient $f(z):=\displaystyle{\prod_{\delta\mid N}\eta_\delta^{r_\delta}}$, if $\displaystyle{\sum_{\delta\mid N}\delta r_\delta}\equiv 0 \pmod{24}$ and $\displaystyle{\sum_{\delta\mid N}\dfrac{N r_\delta}{\delta }}\equiv 0 \pmod{24}$, then for all $\begin{pmatrix}
			a & b\\
			c & d
		\end{pmatrix}\in \Gamma_0(N)$, $f(z)$ satisfies
		\begin{align*}
			f\left(\dfrac{az+b}{cz+d}\right)=\chi(d)(cz+d)^\ell f(z),
		\end{align*}
		where $\ell:=\frac{1}{2}\displaystyle{\sum_{\delta\mid N}r_\delta}$ is an integer and the character $\chi(d):=\displaystyle{\left(\dfrac{(-1)^\ell\prod_{\delta\mid N}\delta^{r_\delta}}{d}\right)}$.
	\end{lemma}
	
	The next lemma helps us check if an $\eta$-quotient $f(z)$ is holomorphic at all the cusps of $\Gamma_0(N)$ and thereby $f(z)\in M_\ell (\Gamma_0(N), \chi)$.
	
	\begin{lemma}\cite[Theorem 1.65]{Ono04}\label{Order of Eta-Quotients} Let $c,d$, and $N$ be positive integers such that $d\mid N$ and $\textup{gcd}(c,d)=1$. If an $\eta$-quotient $f(z)$ satisfies Lemma \ref{Modularity of Eta-Quotients} for $N$, then the order of vanishing of $f(z)$ at the cusp $\dfrac{c}{d}$ is given by
		\begin{align*}
			\dfrac{N}{24}\sum_{\delta\mid N}\dfrac{\textup{gcd}(d,\delta)^2r_\delta}{\textup{gcd}\left(d,\frac{N}{d}\right) d \delta}.
		\end{align*} 
	\end{lemma}
	
	We also need the action of the Hecke operator $T_m$ for positive integer $m$ on $\displaystyle{f(z)=\sum_{n=0}^\infty A(n)q^n}$ $\in$  $M_\ell (\Gamma_0(N),\chi)$ defined by
	\begin{align*}
		f(z)\mid T_m:=\sum_{n=0}^\infty \left(\sum_{d\mid (n,m)}\chi(d)d^{\ell-1}A\left(\dfrac{nm}{d^2}\right)\right)q^n.
	\end{align*}
	If $m=p$ is a prime, then
	\begin{align*}
		f(z)\mid T_p:=\sum_{n=0}^\infty \left(A(pn)+\chi(p)p^{\ell-1}A\left(\dfrac{n}{p}\right)\right)q^n,
	\end{align*}
	where $A(n/p)=0$ if $p\nmid n$.

	At last, we need a lemma, which checks the equivalence of two modular forms modulo a prime.
	\begin{lemma}\cite{Stu87}\label{Equivalence of Eta-Quotients} 
		Let $f(z):=\displaystyle{\sum_{n=n_0}^\infty x(n)q^n}$ and $g(z):=\displaystyle{\sum_{n=n_1}^\infty y(n)q^n}$ be two modular forms in  $M_{\ell}(\Gamma_0(N),\chi)$, where $n_0,n_1\ge 0$. For a prime $p$, if $x(n)\equiv y(n) \pmod{p}$ for all $n\le \nu$, where $\nu=\dfrac{\ell N}{12}\displaystyle{\prod_{d\ prime;\ d\mid N}\left(1+\dfrac{1}{d}\right)}$, then $x(n)\equiv y(n) \pmod{p}$ for all $n\ge0$.
	\end{lemma}

	\subsection{Radu's Algorithm}\label{sec:rk}This subsection explains Radu's algorithm which we will use later to prove some results. We use Smoot's \cite{Smoot} implementation of an algorithm of Radu \cite{Radu}, which we describe now. Radu's algorithm can be used to prove Ramanujan type congruences of the form stated in Introduction. The algorithm takes as an input the generating function
	\[
	\sum_{n= 0}^\infty a_r(n)q^n=\prod_{\delta|M}\prod_{n= 1}^\infty (1-q^{\delta n})^{r_\delta},
	\]
	and positive integers $m$ and $N$, with $M$ another positive integer and $(r_\delta)_{\delta|M}$ is a sequence indexed by the positive divisors $\delta$ of $M$. With this input, Radu's algorithm tries to produce a set $P_{m,j}(j)\subseteq \{0,1,\ldots, m-1\}$ which contains $j$ and is uniquely defined by $m, (r_\delta)_{\delta|M}$ and $j$. Then, it decides if there exists a sequence $(s_\delta)_{\delta |N}$ such that
	\[
	q^\alpha \prod_{\delta|M}\prod_{n=1}^\infty (1-q^{\delta n})^{s_\delta} \cdot \prod_{j^\prime \in P_{m,j}(j)}\sum_{n=0}^\infty a(mn+j^\prime)q^n,
	\]
	is a modular function with certain restrictions on its behaviour on the boundary of $\mathbb{H}$.
	
	Smoot \cite{Smoot} implemented this algorithm in Mathematica and we use his \texttt{RaduRK} package, which requires the software package \texttt{4ti2}. Documentation on how to intall and use these packages are available from Smoot \cite{Smoot}. We use this implemented \texttt{RaduRK} algorithm to find the generating function of $\eou(14n+8)$ and some congruences modulo 64.
	
	It is natural to guess that $N=m$ (which corresponds to the congruence subgroup $\Gamma_0(N)$), but this is not always the case, although they are usually closely related to one another. The determination of the correct value of $N$ is an important problem for the usage of \texttt{RaduRK} and it depends on a criterion called the $\Delta^\ast$ criterion \cite[Definitions 34 and 35]{Radu}, which we do not explain here. It is easy to check the minimum $N$ which satisfies this criterion by running \texttt{minN[M, r, m, j]}, which we do now for our generating function. The generating function of $\eou(n)$ given in \eqref{gf:eou} can be described by setting $M=4$ and $r=\{0,-2,2\}$. We are interested in only the case $m=14$ and $j=8$ for finding the generating function of $\eou(14n+8)$. Running \texttt{minN[4, \{0,-2,2\}, 14, 8]} yields $56$ as the output for the minimum choices of $N$. This value of $N$ is large enough for the computation to take quite a while on a modest laptop. So, we are going to use the generating function for $\eou(2n)$ to find the generating function of $\eou(14n+8)$ and congruences in Theorem \ref{thm:cong64}.
	
	In the next three sections, we present our results and their proofs alongside them.
	
	\section{Elementary Approach}\label{Elementary Approach}
	In this section, we obtain our results using $q$-series. We begin by confirming Conjecture \ref{conj}.
	\begin{theorem}\label{thm-2}
		Conjecture \ref{conj} is true.
	\end{theorem}
	\begin{proof}
		We first prove \eqref{eob:cong-conj}.
		From \eqref{gf:eob},
		\begin{equation*}
			\sum_{n=0}^\infty \eob(n)q^n\equiv f_2^4 \pmod 2,
		\end{equation*}
		which gives
		\begin{equation}\label{eq:h-2}
			\sum_{n=0}^\infty \eob(2n)q^n\equiv f_1^4 \pmod 2.
		\end{equation}
		
		Applying \eqref{5-diss} in \eqref{eq:h-2} and then extracting the terms involving powers of $q$ of the types $5n+1$ and $5n+2$, we get
		\begin{align*}
			\sum_{n=0}^\infty  \eob(10n+2)q^n &\equiv f_5^4\left(-\frac{4}{R^3(q)}+2R^2(q)\right) \pmod 2,\\
			\sum_{n=0}^\infty  \eob(10n+4)q^n &\equiv f_5^4\left(\frac{2}{R^2(q)}+4R^3(q)\right) \pmod 2.
		\end{align*}
		These two imply to \eqref{eob:cong-conj}.
		
		We now prove \eqref{eou:cong-conj}. By \eqref{bbo1}, we have
		\begin{equation}\label{eq:eo1}
			\sum_{n=0}^\infty \eou(2n)q^n=\frac{f_2^2}{f_1^2}=\frac{f_{10}^2}{f_5^6}\left(\psi^4(q)-2q\psi^2(q)\psi^2(q^5)+q^2\psi^4(q^5) \right).
		\end{equation}
		Using \eqref{eq:ndb-1},  \eqref{5-diss}, and \eqref{5-diss-psi} in \eqref{eq:eo1}, we obtain
		\begin{align}
			\label{eq:gf-2n} \sum_{n=0}^\infty \eou(2n)q^n &= \frac{f_{10}^2}{f_5^6}\left( f_1^2f_5^2+4q\psi^2(q)\psi^2(q^5)-4q^2\psi^4(q^5)\right)\notag\\
			& = \frac{f_{10}^2}{f_5^6}\Bigg( -4q^2\psi^4(q^5)+4q\psi^2(q^5)\left(f\left(q^{10},q^{15}\right)+qf\left(q^5,q^{20}\right)+q^3\psi\left(q^{25}\right)\right)^2\notag\\
			& \qquad \qquad +f_5^2f_{25}^2\left(\frac{1}{R(q^5)}-q-q^2R(q^5) \right)^2\Bigg).
		\end{align}
		From \eqref{eq:gf-2n}, extracting the terms that involve $q^{5n+1}$ and $q^{5n+3}$, we have
		\begin{align*}
			\sum_{n=0}^\infty \eou(10n+2)q^n&=2\frac{f_2^2}{f_1^6}\Bigg(2\psi^2(q)f^2\left(q,q^4\right)-\frac{f_1^2f_5^2}{R(q)} \Bigg),\\
			\sum_{n=0}^\infty \eou(10n+6)q^n&=2\frac{f_2^2}{f_1^6}\left(2\psi^2f^2\left(q^2,q^3\right)+f_1^2f_5^2R(q)\right),
		\end{align*}
		which together prove \eqref{eou:cong-conj}.
	\end{proof}

	With the aid of Lemma \ref{lem:2}, we can arrive at the following result handily.

	\begin{theorem}\label{thm:cong-p}
		Let $g\geq 1$ and let $p=6g\pm1$. For all $n\geq0$ and
		$0\leq i\leq p-1$, we have
		\[
		\eou(2pn+2i)\equiv0\pmod2,
		\]
		provided that
		\(
		i\not\equiv \frac{p^2-1}{12}\pmod p
		\)
		and
		\(
		i\not\equiv (k-g)(3k-3g\mp1)\pmod p
		\)
		for every $1\leq k\leq (p-1)/2$, where the signs are paired.
		Thus, the minus sign is used in the second expression when
		$p=6g+1$, and the plus sign is used when $p=6g-1$.
	\end{theorem}
	
	\begin{proof}
		Modulo $2$, we have
		\begin{align}
			\label{EOu 2n mod 2}
			\sum_{m=0}^{\infty}\eou(2m)q^m
			=\frac{f_2^2}{f_1^2}
			\equiv f_1^2\pmod2.
		\end{align}
		
		Write $p=6g+\varepsilon$, where $\varepsilon\in\{-1,1\}$.
		By Lemma~\ref{lem:2},
		\[
		f_1^2\equiv f_{p^2}^2
		\left(
		q^{(p^2-1)/12}
		+
		\sum_{k=1}^{(p-1)/2}
		q^{(k-g)(3k-3g-\varepsilon)}
		\frac{
			f(-q^{2pk},-q^{p^2-2pk})^2
		}{
			f(-q^{pk},-q^{p^2-pk})^2
		}
		\right)
		\pmod2.
		\]
		Now, $f_{p^2}^2$ is a power series in $q^p$. Moreover, for
		each $k$,
		\(
		\dfrac{
			f(-q^{2pk},-q^{p^2-2pk})^2
		}{
			f(-q^{pk},-q^{p^2-pk})^2
		}
		\)
		is also a power series in $q^p$. It follows that, modulo $2$, the only residue classes modulo
		$p$ that can occur among the exponents of $f_1^2$ are
		\(
		\frac{p^2-1}{12}\pmod p
		\)
		and
		\(
		(k-g)(3k-3g-\varepsilon)\pmod p,\)
		for \( 1\leq k\leq\frac{p-1}{2}.
		\)
		Consequently, if $i$ does not belong to any of these residue
		classes, the coefficient of $q^{pn+i}$ in $f_1^2$ is zero
		modulo $2$. Comparing coefficients of $q^{pn+i}$ in
		\eqref{EOu 2n mod 2} gives
		\[
		\eou\bigl(2(pn+i)\bigr)
		=\eou(2pn+2i)\equiv0\pmod2,
		\]
		as required.
	\end{proof}

	Due to Theorem \ref{thm:cong-p}, we also have the following congruences.
	\begin{theorem}
		For all $n\ge0$, $k\ge0$, and primes $p=6g\pm1$, $g\ge1$, we have
		\begin{align}
			\label{SS mod 2}\eou\left(4\left(p^{2k+2}n+p^{2k+1}r+\dfrac{p^{2k+2}-1}{24}\right)\right)\equiv 0 \pmod{2}, \quad 1\le r\le p-1.
		\end{align}
	\end{theorem}
	\begin{proof}
		From \eqref{EOu 2n mod 2}, we have
		\begin{align*}
			\sum_{n=0}^\infty \eou(4n)q^n&\equiv f_1 \pmod{2}.
		\end{align*}
		For primes $p=6g\pm 1$, it is not tough to evaluate that $(p^2-1)/24\neq (k-g)(3k-3g\mp 1)/2$ for all $1\le k\le (p-1)/2$. So, we invoke the $p$-dissection of $f_1$ in the above identity and then extract the terms involving $q^{pn+(p^2-1)/24}$ to obtain
		\begin{align*}
			\sum_{n=0}^\infty \eou\left(4\left(pn+\dfrac{p^2-1}{24}\right)\right)q^n&\equiv f_p \equiv \sum_{n=0}^\infty \eou(4n)q^{pn}\pmod{2}.
		\end{align*}
		The above congruence is a self-similarity enjoyed by $\eou(n)$. This self-similarity result, on induction, gives
		\begin{align*}
			\sum_{n=0}^\infty \eou\left(4\left(p^{2k+1}n+\dfrac{p^{2k+2}-1}{24}\right)\right)q^n&\equiv f_p\pmod{2}
		\end{align*}
		for all $k\ge 0$. The above identity immediately proves \eqref{SS mod 2} as the right side produces not term of the form $q^{pn+r}$, $1\le r\le p-1$.
	\end{proof}
	
	We next prove the following exact generating function.
	\begin{theorem}\label{thm:gf-14-8}
		We have
		\begin{align}
			\label{Gen 10n+4}\sum_{n=0}^\infty \eou\left(10n+4\right)q^n &=3\dfrac{ f_2^4f_5^6}{f_1^8 f_{10}^2}+8q\dfrac{f_2^3  f_5f_{10}^3}{f_1^7}.
		\end{align}
	\end{theorem}
	\begin{proof}
		From \eqref{eq:gf-2n}, we obtain
		\begin{align*}
			\sum_{n=0}^\infty \eou(2n)q^n=  \frac{f_{10}^2}{f_5^6} \Bigg(-4q^2\psi^4(q^5)+4q\psi^2(q^5)\{A^2(q^5)+q^2B^2(q^5)+q^6\psi^2(q^{25})\\\qquad \qquad \qquad  +2qA(q^5)B(q^5)+2q^3A(q^5)\psi(q^{25})+2q^4B(q^5)\psi(q^{25})\}\\ +f_5^2f_{25}^2\left(\frac{1}{R^2(q^5)}+q^4R^2(q^5)-\frac{2q}{R(q^5)} +2q^3R(q^5)\right)\Bigg),
		\end{align*}
		where $A(q)=f(q^2,q^3)$ and $B(q)=f(q,q^4)$.\\
		
		Expanding the right side of the above identity and then extracting the terms that involve $q^{5n+2}$, we have
		\begin{align*}
			\sum_{n=0}^\infty \eou(10n+4)q^n & = \frac{f_2^2}{f_1^6}\left( -4\psi^4(q)+4q\psi^2(q)\psi^2(q^5)+8\psi^2(q)f\left(q,q^4\right)f\left(q^2,q^3\right)-f_1^2f_5^2\right),
		\end{align*}
		which by \eqref{bbo1} reduces to
		\begin{align*}
			\sum_{n=0}^\infty \eou(10n+4)q^n &  = \frac{f_2^2}{f_1^6}( -4\psi^4(q)+4q\psi^2(q)\psi^2(q^5)+8\psi^2(q)\left(\psi^2(q)-q\psi^2(q^5)\right)-f_1^2f_5^2)\\
			&=  \frac{f_2^2}{f_1^6}( 4\psi^2(q)\left(\psi^2(q)-q\psi^2(q^5)\right)-f_1^2f_5^2)\\
			&=  4\frac{f_2^7f_5^3}{f_1^9f_{10}}-\frac{f_2^2f_5^2}{f_1^4}.
		\end{align*}
		Simplifying this further and using \eqref{eq:nilufar-1}, we now obtain
		\begin{align}
			\label{eq:ndb-gf-1}
			\sum_{n=0}^\infty \eou(10n+4)q^n&=3\frac{f_2^2f_5^2}{f_1^4}+20q\frac{f_2^3f_5f_{10}^3}{f_1^7}\\
			\label{eq:thm-10n-4} &=3\frac{ f_2^4f_5^6}{f_1^8 f_{10}^2}+8q\frac{f_2^3  f_5f_{10}^3}{f_1^7}.
		\end{align}
		This completes the proof.
	\end{proof}
	
	Using the above generating functions, we prove the following congruences.
	\begin{corollary}\label{thm:inf-1}
		For all $n\ge0$ and $k\ge0$, we have
		\begin{align}
			\label{Internal Cong mod 5 1}\eou\left(2\times 5^{2 k+3}n+\dfrac{5^{2k+4}-1}{6}\right)&\equiv3^{k+1}\eou(10n+4) \pmod 5,\\
			\label{Inf fam mod 8 cong 1}  \eou\left(2\times 5^{2k+2}n+2\times 5^{2k+1}r+\dfrac{5^{2k+2}-1}{6}\right)&\equiv  0 \pmod{16},\quad r\in\{1,2,3,4\}.
		\end{align}
	\end{corollary}
	\begin{proof}
		From \eqref{eq:ndb-gf-1}, we have
		\begin{align}
			\label{Internal Req 1} \sum_{n=0}^\infty \eou(10n+4)q^n\equiv3\frac{f_2^2f_5^2}{f_1^4} \equiv 3f_1f_2^2f_5 \pmod 5.
		\end{align}
		Using \eqref{5-diss} in the above identity and then extracting the terms that involve $q^{5n}$, we find that
		\begin{align*}
			\sum_{n=0}^\infty \eou(50n+4)q^n\equiv 3f_1f_5f_{10}^2\left(\dfrac{1}{R(q)R\left(q^2\right)}+q-q^2R(q)R\left(q^2\right)\right) \pmod 5.
		\end{align*}
		We then invoke the following modular equation from \cite[Lemma 1.3]{NDBBegum}:
		\begin{align*}
			\dfrac{1}{R(q)R\left(q^2\right)}-q^2R(q)R\left(q^2\right)&=\frac{f_2 f_5^5}{f_1 f_{10}^5}
		\end{align*}
		in the previous identity to find that 
		\begin{align*}
			\sum_{n=0}^\infty \eou(50n+4)q^n\equiv 3\dfrac{f_2f_5^6}{f_{10}^3}+3qf_1f_5f_{10}^2 \pmod 5.
		\end{align*}
		The above identity, with the help of by \eqref{5-diss} again, gives
		\begin{align*}
			\sum_{n=0}^\infty \eou(50n+4)q^n&\equiv 3\dfrac{f_5^6f_{50}}{f_{10}^3}\left(\dfrac{1}{R\left(q^{10}\right)}-q^2-q^4R\left(q^{10}\right)\right)\\
			&\quad +3qf_5f_{10}^2f_{25}\left(\dfrac{1}{R\left(q^{5}\right)}-q-q^2R\left(q^{5}\right)\right) \pmod 5.
		\end{align*}
		Now, extracting the terms involving $q^{5n+2}$ from the above identity, we obtain
		\begin{align}
			\label{Intrenal Req 2} \sum_{n=0}^\infty \eou(250n+104)q^n&\equiv -3\dfrac{f_1^6f_{10}}{f_{2}^3}-3f_1f_{2}^2f_{5}\equiv 4f_1f_{2}^2f_{5} \pmod 5.
		\end{align}
		Therefore, \eqref{Internal Req 1} and \eqref{Intrenal Req 2} imply that
		\begin{align*}
			\sum_{n=0}^\infty \eou(250n+104)q^n&\equiv 3\sum_{n=0}^\infty \eou(10n+4)q^n \pmod 5,
		\end{align*}
		which by induction gives \eqref{Internal Cong mod 5 1}.
		
		From \eqref{eq:thm-10n-4} using \eqref{eq:nilufar-2} once, we obtain
		\[
		\sum_{n=0}^\infty \eou(10n+4)q^n=3\frac{f_2^6f_5^{10}}{f_1^{12}f_{10}^4}-4q\frac{f_2^5f_5^5f_{10}}{f_1^{11}}-32q^2\frac{f_2^4f_{10}^6}{f_1^{10}}.
		\]
		Applying \eqref{eq:nilufar-2} once each on the first two summands in the above equation, we have
		\begin{align}
			\sum_{n=0}^\infty \eou(10n+4)q^n&=3\frac{f_2^8f_5^{14}}{f_1^{16}f_{10}^6}-16q\frac{f_2^7f_5^9}{f_1^{15}f_{10}}+16q^2\frac{f_2^6f_5^4f_{10}^4}{f_1^{14}}-32q^2\frac{f_2^4f_{10}^6}{f_1^{10}}\notag\\
			\label{eq:gf-50 1} &\equiv 3\frac{f_2^8f_5^{14}}{f_1^{16}f_{10}^6}\equiv 3\frac{f_5^{14}}{f_{10}^6}\pmod{16}.
		\end{align}
		which immediately gives
		\begin{align}
			\label{Inf fam mod 8 req 1} \eou(50n+10r+4)\equiv  0 \pmod{16}, \quad r\in\{1,2,3,4\}.
		\end{align}
		
		From \eqref{eq:gf-50 1}, we have 
		\begin{align}
			\label{Inf fam mod 8 req 2}\sum_{n=0}^\infty \eou\left(50n+4\right)q^n 
			&\equiv 3\frac{f_2^2}{f_1^2} \equiv 3 \sum_{n=0}^\infty \eou\left(2n\right)q^n\pmod{16}.
		\end{align}
		Induction in \eqref{Inf fam mod 8 req 2} in light of  \eqref{Inf fam mod 8 req 1} immediately gives \eqref{Inf fam mod 8 cong 1}.
	\end{proof}
	
	As a consequence of \eqref{Gen 10n+4}, we now have the following congruences modulo 32.
	\begin{corollary}\label{cor:cong-100}
		For all $n\geq 0$, we have
		\begin{equation}\label{eq:mod32}
			\eou(100n+r)\equiv 0 \pmod{32}, \quad  r\in\{14,34,74,94\}.
		\end{equation}
	\end{corollary}
	\begin{proof}
		From \eqref{eq:thm-10n-4}, we recall that
		\begin{align*}
			\sum_{n=0}^\infty \eou(10n+4)q^n&=3\frac{ f_2^4f_5^6}{f_1^8 f_{10}^2}+8q\frac{f_2^3  f_5f_{10}^3}{f_1^7}.
		\end{align*}
		Applying \eqref{eq:nilufar-2} multiple times on the right side under modulo 32, we arrive at
		\begin{align*}
			\sum_{n=0}^\infty \eou(10n+4)q^n\equiv3\frac{f_{10}^2}{f_5^2}+16q^2 \frac{f_{40} f_{20}}{f_2}\pmod{32}.
		\end{align*}
		Employing \eqref{disf1} on the first term of the above identity and extracting the odd  powered terms of $q$ from both sides, we have
		\begin{align*}
			\sum_{n=0}^\infty \eou(20n+14)q^n\equiv6q^2 \frac{f_{10}^2f_{40}^2}{f_5^3 f_{20}}\pmod{32}.
		\end{align*}
		Since the above congruence has no terms involving $q^{5n+j}$ for $j\in\{0, 1, 3, 4\}$ in the right side, we find \eqref{eq:mod32}.
	\end{proof}
	
	In the next result, we have a recurrence relation for $\eou(n)$.
	\begin{theorem}\label{th:eou:p(n)}
		For all $n\geq0$, we have
		\begin{align}\label{eq:eou:p(n)}
			\eou(2n)=\sum_{k=0} ^{\lfloor\frac{\sqrt{1+8n}-1}{2}\rfloor }p\left(n-\frac{k(k+1)}{2}\right).
		\end{align}
	\end{theorem}
	\begin{proof}
		From \eqref{gf:eou}, we have
		\begin{align}
			\sum_{n=0}^\infty \eou(2n)q^n=\frac{f_{2}^2}{f_{1}^2}=\frac{\psi(q)}{f_{1}}= \left(\sum_{n=0}^\infty q^{\frac{n(n+1)}{2}}\right)\left(\sum_{n=0}^\infty p(n)q^{n}\right)=\sum_{n=0}^\infty \sum_{k=0}^\infty p(n)q^{n+\frac{k(k+1)}{2}},\label{eoupn}
		\end{align}
		where we have invoked \eqref{p(n)} and \eqref{psi}.  Comparing the coefficients of $q^n$ on both sides \eqref{eoupn}, we easily arrive at \eqref{eq:eou:p(n)}.
	\end{proof}

	\section{Modular Forms Approach}\label{Modular Approach}
	In this section, we present some self-similarities of $\eou(n)$ and their consequences. It is worth noting that \eqref{eq:gf-50 1} is also a self-similarity of $\eou(n)$ modulo $16$, which does not come through the modular form approach. However, modular forms are paramount in finding multiple self-similarities modulo $4$. We state such results in the following theorem.

	\begin{theorem}\label{Thm SS EOu 4n+2}
		We have
		\begin{align}
			\label{SS EOu 4n+2 mod 4} \sum_{n=0}^\infty \eou (4pn+4r+2)q^n\equiv q^m\sum_{n=0}^\infty \eou (4n+2)q^{pn}\pmod{4}
		\end{align}
		for $(p,r,m)=(5,3,2), $ $ (7,5,3), (11,10,5), (17,3,9), $ $ (19,5,10), (23,10,12),$ and $(29,20,15)$.
	\end{theorem}
	
	\begin{proof}
		We prove \eqref{SS EOu 4n+2 mod 4} for the tuple $(5,3,2)$ only. The proofs for the other tuples come along the same vein after constructing relevant modular forms. We first 4-dissect \eqref{gf:eou} using \eqref{disf1}  to obtain
		\begin{align}
			\label{Gen EOu 4n+2} \sum_{n=0}^\infty \eou (4n+2)q^n&=2\dfrac{ f_2^2 f_8^2}{f_1^3 f_4}\equiv 2 f_1f_4^3\pmod{4}.
		\end{align}

		Now, we take the following auxiliary functions:
		\begin{align}
			G_5&:=\eta _1 \eta _4^3 \eta _5^3 \eta _{20}\equiv \dfrac{f_5^3 f_{20}}{2} \sum_{n=0}^\infty \eou (4n+2)q^{n+2}\pmod{2},\\
			H_5&:=\eta _1^3 \eta _4 \eta _5 \eta _{20}^3\equiv q^3\dfrac{f_1^3 f_4}{2} \sum_{n=0}^\infty \eou (4n+2)q^{5n}\pmod{2}.
		\end{align}
		Using Lemmas \ref{Modularity of Eta-Quotients} and \ref{Order of Eta-Quotients}, we derive that $G_5$ and $H_5$ both belong to $\displaystyle{M_4\left(\Gamma_0(20),\left(\dfrac{100}{\bullet}\right)\right)}$. Now, we apply the Hecke operator $T_5$ on $G_5$, which gives
		\begin{align}
			\label{G_5|T_5} G_5\mid T_5\equiv \dfrac{f_1^3 f_{4}}{2} \sum_{n=0}^\infty \eou (20n+14)q^{n+1}\pmod{2}.
		\end{align}
		
		Now, the endomorphicity of $T_5$ on $\displaystyle{M_4\left(\Gamma_0(20),\left(\dfrac{100}{\bullet}\right)\right)}$ ensures that $G_5\mid T_5$ also belongs to the space $ \displaystyle{M_4\left(\Gamma_0(20),\left(\dfrac{100}{\bullet}\right)\right)}$. Now we employ Lemma \ref{Equivalence of Eta-Quotients} on $G_5\mid T_5$ and $H_5$. Using Mathematica, we verify the following identity:
		\begin{align}
			\label{G5|T5 and H5} G_5\mid T_5\equiv \dfrac{f_1^3 f_{4}}{2} \sum_{n=0}^\infty \eou (20n+14)q^{n+1} \equiv q^3\dfrac{f_1^3 f_4}{2} \sum_{n=0}^\infty \eou (4n+2)q^{5n}\equiv H_5 \pmod{2}
		\end{align}
		is true up to the coefficient of $q^\nu$, where $\nu=12$ is the Sturm bound for $G_5\mid T_5$ and $H_5$ on the space $ \displaystyle{M_4\left(\Gamma_0(20),\left(\dfrac{100}{\bullet}\right)\right)}$. Thus, Lemma \ref{Equivalence of Eta-Quotients} proves \eqref{G5|T5 and H5} for all the coefficients of $q^n$. Cancelling the redundant factors from both sides of \eqref{G5|T5 and H5}, we obtain \eqref{SS EOu 4n+2 mod 4} for $(5,3,2)$. In the table below, we provide the modular forms and their spaces for proving \eqref{SS EOu 4n+2 mod 4} for the respective tuples, which was verified with a Mathematica calculation.
		
		\begin{table}[h]\label{table 1H}
			\renewcommand{\arraystretch}{2.1}
			\caption{Modular forms for proving \eqref{SS EOu 4n+2 mod 4}}
			\begin{tabular}{ c  c  c  c  c }
				\hline 
				$(p,r,m)$  & $G_p$  & $H_p$  &  $M_\ell (\Gamma_0(N),\chi(\bullet))$ & $\nu$ \K\D\\  \hline\vspace{1mm}
				$(7,5,3)$ & $\eta _1 \eta _4^3 \eta _7 \eta _{28}$ & $\eta _1 \eta _4 \eta _7 \eta _{28}^3$ &  $M_3 \left(\Gamma_0(112),\left(\dfrac{-196}{\bullet}\right)\right)$ &  48 \\
				\hline
				$(11,10,5)$ & $\dfrac{\eta _1 \eta _4^3 \eta _{11}^5}{\eta _{44}}$ & $\dfrac{\eta _1^5 \eta _{11} \eta _{44}^3}{\eta _4}$ &  $M_4 \left(\Gamma_0(44),\left(\dfrac{484}{\bullet}\right)\right)$  & 24 \vspace{1mm}\\
				\hline
				$(17,3,9)$ & $\dfrac{\eta _1 \eta _4^3 \eta _{68}^5}{\eta _{17}}$ & $\dfrac{\eta _4^5 \eta _{17} \eta _{68}^3}{\eta _1}$ & $M_4 \left(\Gamma_0(68),\left(\dfrac{1156}{\bullet}\right)\right)$  &  36 \vspace{1mm}\\
				\hline
				$(19,5,10)$ & $\eta _1 \eta _4^3 \eta _{19}^5 \eta _{76}^3$ & $\eta _1^5 \eta _4^3 \eta _{19} \eta _{76}^3$ & $M_6 \left(\Gamma_0(152),\left(\dfrac{1444}{\bullet}\right)\right)$  & 120 \vspace{1mm}\\
				\hline
				$(23,10,12)$ & $\eta _1 \eta _4^3 \eta _{23} \eta _{92}^3$ & $\eta _1 \eta _4^3 \eta _{23} \eta _{92}^3$ &  $M_4 \left(\Gamma_0(92),\left(\dfrac{2116}{\bullet}\right)\right)$  & 48 \vspace{1mm}\\
				\hline
				$(29,20,15)$ & $\eta _1 \eta _4^3 \eta _{29}^3 \eta _{116}$ & $\eta _1^3 \eta _4 \eta _{29} \eta _{116}^3$ &  $M_4 \left(\Gamma_0(116),\left(\dfrac{3364}{\bullet}\right)\right)$ & 60 \vspace{1mm}\\
				\hline
			\end{tabular}
		\end{table}
	\end{proof}
	
	While there is no tuple $(p,r,m)$ with $p=13$ that satisfies \eqref{SS EOu 4n+2 mod 4}, it is to be noted that the primes $p$ of the tuples $(p,r,m)$ in the theorem above are under $30$. It appears that \eqref{SS EOu 4n+2 mod 4} is true for other tuples with higher primes as well. In this regard, we pose a question in Section \ref{sec:conc}. However here, as immediate consequences of Theorem \ref{Thm SS EOu 4n+2}, we have a corollary below.
	\begin{corollary}\label{Cor SS Mod 4}
		For all $n\ge0$ and $k\ge1$, we have
		\begin{align*}
			\eou\left( 4 p^{2k}n+4 p^{2k-1}s+  \dfrac{4  p^{2 k} r+4  p^{2 k-1}m+2 p^2-4  p m-4 r-2}{p^2-1}\right) &\equiv 0 \pmod{4}
		\end{align*}
		for $(p,r,m)=(5,3,2), $ $ (7,5,3), (11,10,5), (17,3,9), $ $ (19,5,10), (23,10,12),$ $(29,20,15)$ and $0\le s <p$, $s\neq m$.
	\end{corollary}
	
	\begin{proof}
		If, we employ induction on $k$ for all the tuples in \eqref{SS EOu 4n+2 mod 4}, we obtain
		\begin{align*}
			\sum_{n=0}^\infty \eou \left( 4 p^{2k-1}n+  \dfrac{4  p^{2 k} r+4  p^{2 k-1}m+2 p^2-4  p m-4 r-2}{p^2-1}\right)q^n&\equiv q^m\sum_{n=0}^\infty \eou (4n+2)q^{pn}\\
			&\quad \pmod{4}.
		\end{align*}
		Now, since no term contains $q^{pn+s}$, $s\neq m$ on the right side of the above identity, we arrive at Corollary \ref{Cor SS Mod 4} from the above identity.   
	\end{proof}
	
	\section{Algorithmic Approach}\label{Algorithmic Approach}
	In this section, we present an exact generating function of $\eou(14n+8)$, which we obtain by Radu's algorithm elaborated in Subsection \ref{sec:rk}. We also present an infinite family of congruences modulo 16. 
	\begin{theorem}\label{thm:gf-10-4}
		We have
		\begin{align}
			\sum_{n= 0}^\infty\eou(14n+8)q^n&=9\frac{f_2^8 f_7^{14}}{f_1^{16} f_{14}^6}-16 q\frac{f_2^{14} f_7^4}{f_1^{18}}-48 q^2\frac{f_2^7 f_7^7 f_{14}}{f_1^{15}}+128 q^3\frac{f_2^{13} f_{14}^7}{f_1^{17} f_7^3}\notag\\
			\label{Gen 14n+8} &\quad-176 q^4\frac{f_2^6 f_{14}^8}{f_1^{14}}-128 q^6\frac{f_2^5 f_{14}^{15}}{f_1^{13} f_7^7}.
		\end{align}
	\end{theorem}
	\begin{proof} We use \eqref{eq:eo1} and calculate \texttt{minN[2,\{-2,2\},7,4]}, which gives $N=14$, which is easily handled in a modest laptop. Radu's algorithm now gives a straight proof of \eqref{Gen 14n+8}. Here we give the output of \texttt{RK}. 
		\allowdisplaybreaks{
			\begin{doublespace}
				\begin{align*}
					\texttt{In[1] := } & \texttt{RaduRK[14,2,\{-2,2\},7,4]}\\
					& \prod_{\texttt{$\delta$|M}} (\texttt{q}^{\delta };\texttt{q}^{\delta })_{\infty }^{\texttt{r}_{\delta }}  = \sum_{\texttt{n=0}}^{\infty } \texttt{a}(\texttt{n})\,\texttt{q}^\texttt{n}\\
					& \fbox{$\texttt{f}_\texttt{1}(\texttt{q})\cdot \prod\limits_{\texttt{j}'\in \texttt{P}_{\texttt{m,r}}(\texttt{j}) } \sum\limits_{\texttt{n=0}}^\infty \texttt{a}(\texttt{mn}+\texttt{j}')\,\texttt{q}^\texttt{n} = \sum\limits_{\texttt{g}\in \texttt{AB}} \texttt{g}\cdot \texttt{p}_\texttt{g}(\texttt{t}) $} \\
					& \texttt{Modular Curve: }\texttt{X}_\texttt{0}(\texttt{N}) \\
					\texttt{Out[2] = }\\
					&\begin{array}{c|c}
						\text{N:} & 14 \\
						\hline
						\text{$\{$M,(}r_{\delta })_{\delta |M}\text{$\}$:} & \{2,\{-2,2\}\} \\
						\hline
						\text{m:} & 7 \\
						\hline
						P_{m,r}\text{(j):} & \{4\} \\
						\hline
						f_1\text{(q):} & \frac{f_1^{13} f_7^7}{q^6f_2^5 f_{14}^{15} } \\
						\hline
						\text{t:} & \frac{f_2 f_7^7}{ q^2f_1 f_{14}^7} \\
						\hline
						\text{AB:} & \left\{1,\frac{f_2^8 f_7^4}{q^3f_1^4 f_{14}^8 }-4\frac{f_2 f_7^7}{ q^2f_1 f_{14}^7}\right\} \\
						\hline
						\left\{p_g\text{(t): g$\in $AB$\}$}\right. & \left\{9 t^3-112 t^2+336 t-128,128-16 t\right\} \\
						\hline
						\text{Common Factor:} & \text{None} \\
					\end{array}
				\end{align*}
		\end{doublespace}  }
	\end{proof}
	
	The interested reader can refer to \cite{Saikia} or \cite{AndrewsPaule2} for more explanation of the method and how to read the output. As a consequence of \eqref{Gen 14n+8}, we present the following congruence..
	\begin{corollary}
		For all $n\ge0$ and $k\ge0$, we have
		\begin{align}
			\label{Inf fam mod 16 prime 7}\eou\left(2\times 7^{2k+2}n+2\times 7^{2k+1}r+\dfrac{7^{2k+2}-1}{6}\right)&\equiv  0 \pmod{16},\quad r\in\{1,2,3,4,5,6\}.
		\end{align}
	\end{corollary}
\begin{proof}
	From \eqref{Gen 14n+8}, we have
\begin{align}
\sum_{n= 0}^\infty\eou(14n+8)q^n&\equiv9\frac{f_2^8 f_7^{14}}{f_1^{16} f_{14}^6}\nonumber\\
&\equiv9\frac{f_{14}^2}{f_7^2}\pmod{16}.\label{14n+8_iterate}
\end{align}
Comparing the coefficients of the terms involving $q^{7n+r}$ for $r\in\{1,2,3,4,5,6\}$, we obtain
\begin{align}\label{cong:infinitefam}
	\eou(98n+14r++8)\equiv0\pmod{16}.
\end{align}
From \eqref{14n+8_iterate}, we have
\begin{align*}
\sum_{n= 0}^\infty\eou(14n+8)q^n&\equiv9\frac{f_2^8 f_7^{14}}{f_1^{16} f_{14}^6}\\
&\equiv \eou(2n)q^{7n}.
\end{align*}
Comparing coefficients pf $q^{7n}$ on both sides of the above, we arrive at
\begin{align*}
	\eou(2\cdot7^2n+8)\equiv\eou(2n)\pmod{16}.
\end{align*}
Iterating the above, we obtain
\begin{align*}
	\eou\left(2\cdot7^{2k}n+\dfrac{7^{2k+2}-1}{6}\right)\equiv\eou(2n)\pmod{16}.
\end{align*}
Replacing $n$ by $7^2n+7r+4$ for $r\in\{1,2,3,4,5,6\}$ and then invoking \eqref{cong:infinitefam}, we readily arrive at \eqref{Gen 14n+8}.
\end{proof}
	Our next result gives a set of congruences modulo $64$.
	\begin{theorem}\label{thm:cong64}
		For all $n\geq 0$, we have
		\begin{equation}\label{eq:mod64}
			\eou(98n+r)\equiv 0 \pmod{64},  \quad  r\in\{22, 36, 64, 50, 78, 92\}.
		\end{equation}
	\end{theorem}
	\begin{proof}[Sketch of the proof]
		We again use \eqref{eq:eo1}, and calculate \texttt{minN[2,\{-2,2\},49,11]}, which gives $N=14$. Running \texttt{RK[14,2,\{-2,2\},49,11]} will give a proof of the first three congruences, namely when $r=22, 36,$ and $ 64$. Further, running \texttt{RK[14,2,\{-2,2\},49,25]} will give a proof of the last three congruences, namely when $r=50, 78,$ and $ 92$. The interested reader can check the outputs available  \href{https://manjilsaikia.in/publ/mathematica/EOu.nb}{here}.
	\end{proof}
	
	\section{Concluding Remarks}\label{sec:conc}
	\begin{enumerate}
		
		\item As a consequence of Theorem \ref{thm-2}, we now have the following congruence via Goswami and Jha \cite[Remark 5.5]{GoswamiJha}
		\[
		\eob(250n+58)\equiv \eob(250n+108)\equiv 0 \pmod 2.
		\]
		and the following congruence via Goswami and Jha \cite[Remark 6.10]{GoswamiJha}
		\[
		\eou(250n+154)\equiv \eou(250n+54)\equiv 0 \pmod 2.
		\]
		It would be interesting to find independent justifications of these congruences.
		\item Congruences \eqref{Inf fam mod 8 req 1} and \eqref{eq:mod32} seem to be parts of a larger family of congruences. We leave it as an open problem to find these family of congruences.
		\item  We leave another problem to find all the tuples $(p,r,m)$, which fit in \eqref{SS EOu 4n+2 mod 4}.
		\item The proofs of \eqref{Gen 14n+8} and Theorem \ref{thm:cong64} are via an algorithmic process, it would be interesting to find elementary proofs for them.
		\item The pattern of \eqref{Inf fam mod 8 cong 1} and \eqref{Inf fam mod 16 prime 7} seems to carry forward. We make the following conjecture.
		\begin{conjecture}
			For all  $n\ge0$, $k\geq 0$ and $1\leq r\leq 10$, we have
			\[
			\eou\left(2\times 11^{2k+2}n+2\times 11^{2k+1}r+\dfrac{11^{2k+2}-1}{6}\right)\equiv  0 \pmod{16}.
			\]
		\end{conjecture}
        \item There is also a more general conjecture that we make.
        \begin{conjecture}
            For all $n, k\geq 0$, prime $p>3$ with $p\equiv 3 \pmod 4$ and $1\leq r\leq p-1$, we have
            \[
\eou\left(2p^{2k+2}n+2p^{2k+1}r+\frac{p^{2k+2}-1}{6}\right)\equiv 0 \pmod{64}.
            \]
        \end{conjecture}
                \item We also make the following related self-similarity conjecture.
        \begin{conjecture}
            For primes $p>3$ with $p\equiv 3\pmod 4$, we define $\lambda_p:=p^2+3p+3$, then we have, for all $n, k\geq 0$
            \[
            \eou\left(2p^{2k}n+\frac{p^{2k}-1}{6}\right)\equiv \lambda_p^k\eou(2n) \pmod {64}. 
            \]
        \end{conjecture}
		\item In this paper, we have found new congruences modulo $2^i$ for $i\in\{1,2,4,5,6\}$ for $\eou(n)$. It seems that $\eou(n)$ satisfies congruences modulo higher powers of $2$ as well, which the interested readers may investigate.
		\item Note that a result similar to Ramanujan's recurrence for the partition function \cite[p.108]{BerndtIV}:
		\[
		np(n)=\sum_{i=1}^n\sigma(i)p(n-i),
		\]
		where $\sigma(i)$ is the divisor-sum function, can also be found for $\eou(n)$ by using the general technique of Bal and Bhatnagar \cite{gaurav}. We leave that as an exercise for the interested reader. Several other recurrences (not listed above) also seem to follow as corollaries of results of Bal and Bhatnagar \cite{gaurav}.
	\end{enumerate}
	
	\section*{Acknowledgements}
	This work was completed in the summer of 2024, when we submitted the work in 2026, the third author asked OpenAI's ChatGPT 5.6 Sol to review the paper. ChatGPT caught a few errors in the original draft, which we have since corrected. No new results or proof ideas were suggested by ChatGPT, all the mathematical content and ideas came from the authors.

	\bibliographystyle{alpha}

\end{document}